\documentclass[11pt]{amsart} 

\usepackage{amssymb} 
\usepackage{enumerate} 
\usepackage{hyperref} 
\usepackage{tikz} 
\usepackage{xcolor} 
\usepackage{amsmath} 
\usepackage{float} 
\usepackage{mathtools}
\usepackage{tikz-cd}
\usepackage{graphicx}
\usepackage[backend=biber,style=numeric]{biblatex}
\usepackage[T1]{fontenc}
\usepackage[twoside,includefoot,footskip=25pt, margin=1.2in]{geometry}
\newcommand{\R}{\mathbb R} 

\makeatletter
\newcommand{\tpitchfork}{%
  \vbox{
    \baselineskip\z@skip
    \lineskip-.52ex
    \lineskiplimit\maxdimen
    \m@th
    \ialign{##\crcr\hidewidth\smash{$-$}\hidewidth\crcr$\pitchfork$\crcr}
  }%
}
\makeatother

\makeatletter
\renewcommand*\env@matrix[1][*\c@MaxMatrixCols c]{%
  \hskip -\arraycolsep
  \let\@ifnextchar\new@ifnextchar
  \array{#1}}
\makeatother

\DeclareMathOperator{\tb}{tb}

\theoremstyle{plain} 
\newtheorem{theorem}{Theorem}[section]
\newtheorem{lemma}[theorem]{Lemma}

\theoremstyle{definition} 
\newtheorem{definition}[theorem]{Definition} 
 
\newtheorem{example}[theorem]{Example}

\title{Legendrian Reidemeister Moves for the Convex Surface Projection}
\author[Zijian Rong]{Zijian Rong}
\address{Department of Mathematics\\University of Southern California\\Los Angeles, CA\\90007\\USA}
\email{zijianro@usc.edu}
\date{\today}

\begin{document}

\begin{abstract}
    We prove a Legendrian Reidemeister theorem for Legendrian knots in thickened convex surfaces.
\end{abstract}

\maketitle

\section{Introduction}

Let $(M, \xi)$ be a contact $3$-manifold and let $\Sigma \subset M$ be a closed surface. Then $\Sigma$ is \textit{convex} if there exists a contact vector field $X$ on $M$ transverse to $\Sigma$. Flowing along $X$ we get a contact submanifold $\Sigma \times \R \subset M$, whose contact structure is essentially determined by a set of curves $\Gamma := \{p \in \Sigma : X_p \in \xi_p\}$ so we denote it by $\xi_\Gamma$. We are interested in Legendrian knots $\Lambda \subset (\Sigma \times \R, \xi_\Gamma)$. A natural projection to take is $\pi: \Sigma \times \R \rightarrow \Sigma$ given by $(p, t) \mapsto p$, which we call the \textit{convex surface projection}. The main goal of this paper is to prove a Legendrian Reidemeister theorem for this projection. See Figure \ref{Reidemeister}.

\begin{figure}
    \begin{center}
    	\includegraphics[width=1.0\textwidth]{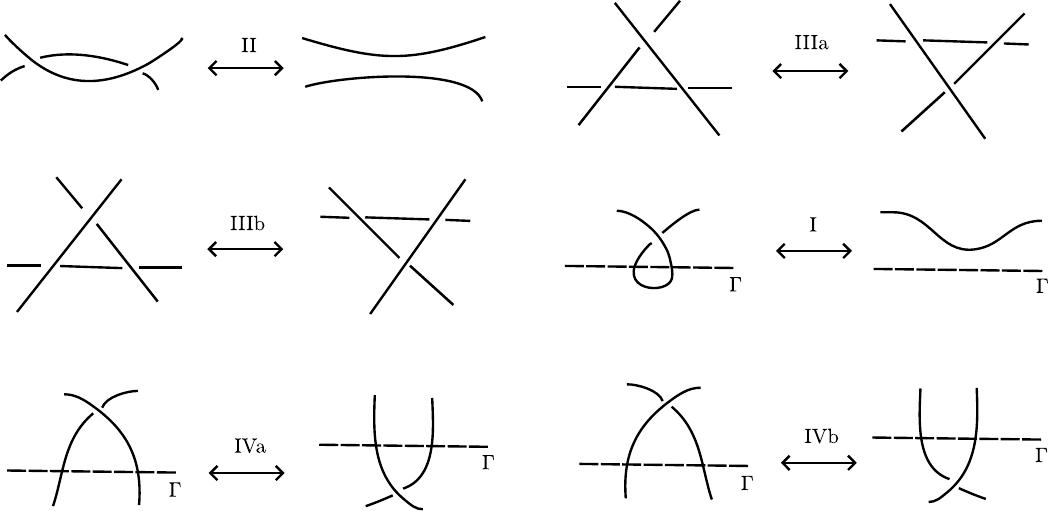}
    	\caption{Reidemeister moves for the convex surface projection}
    	\label{Reidemeister}
    \end{center}
\end{figure}

For $\R^3$ with the standard contact structure $\ker(dz - ydx)$, one can take the front projection (i.e., the $xz$-projection) and the Lagrangian projection (i.e., the $xy$-projection). There are corresponding Legendrian Reidemeister theorems for both projections, which are used to define knot invariants for Legendrian knots in $\R^3$. In particular, in \cite{chekanov2002dga} Chekanov defines the Chekanov-Eliashberg differential graded algebra, a powerful invariant that can distinguish Legendrian knots which share the same classical invariants. To show that the DGA is indeed an invariant Chekanov used the Legendrian Reidemeister theorem for the Lagrangian projection. Various other types of Legendrian Reidemeister theorems have also been proved and used later on, including \cite{sabloff2003circle} \cite{licata2011lensspaces} \cite{licata2013seifert} \cite{ekholmNg2015weinstein} \cite{Bjorklund2016} \cite{GayLicata2018} \cite{Brand2023}. In \cite{EaglesRong2026}, to define a similar DGA invariant for thickened convex surfaces Eagles and the author also used a Legendrian Reidemeister theorem for the convex surface projection, whose proof is contained in this paper, which was meant to be the first part of a project proposed by Chaidez to extend the computation of Legendrian contact homology to thickened convex surfaces.

The proof is mainly motivated by viewing the part of $\pi(\Lambda)$ away from the dividing curve $\Gamma$ as the Lagrangian projection, and treating the part of $\pi(\Lambda)$ near $\Gamma$ by switching to the projection onto $\Gamma \times \R_t \subset \Sigma \times \R_t$, which is similar to the front projection in $\R^3$.

\begin{definition}
	Let $\Lambda$ be a Legendrian knot in $(\Sigma \times \R, \xi_\Gamma)$. Then $\Lambda$ has a \textit{good} projection if $\pi(\Lambda)$ is an immersed submanifold of $\Sigma$ whose only self-intersections are transverse double intersections off $\Gamma$ and $\pi(\Lambda)$ intersects $\Gamma$ transversely.
\end{definition}

\begin{definition}
	Let $\Lambda_t$ be a smooth isotopy in $(\Sigma \times \R, \xi_\Gamma)$ between two Legendrian knots $\Lambda_0$ and $\Lambda_1$. Then $\Lambda_t$ has a \textit{good} projection if $\pi(\Lambda_t)$ can be discretized into a sequence of Reidemeister-like moves as shown in Figure \ref{Reidemeister}:
	
	(a) Reidemeister I move with bifurcation point on $\Gamma$
	
	(b) Reidemeister II move with bifurcation point off $\Gamma$
	
	(c) Reidemeister III move with bifurcation point off $\Gamma$
	
	(d) Reidemeister III-like move with two strands of $\Lambda$ and $\Gamma$, which we call \textit{Reidemeister IV move}

    (e) Composition of a Reidemeister II move and a Reidemeister IV move (see last row of Figure \ref{Reidemeister IV}).
\end{definition}

Now we state the main theorem of this paper:

\begin{theorem}\label{main}
	Let $(\Sigma, \Gamma)$ be a convex surface. 
	
	(a) Let $\Lambda$ be a Legendrian knot in $(\Sigma \times \R, \xi_\Gamma)$. Then there exists an arbitrarily small Legendrian perturbation $\Lambda'$ of $\Lambda$ such that $\Lambda'$ has a good projection.
	
	(b) Let $\Lambda_t$ be a Legendrian isotopy of Legendrian knots in $(\Sigma \times \R, \xi_\Gamma)$ such that $\Lambda_0$ and $\Lambda_1$ have good projections. Then there exists an arbitrarily small Legendrian perturbation $\Lambda'_t$ of $\Lambda_t$ such that $\Lambda'_0 = \Lambda_0$, $\Lambda'_1 = \Lambda_1$, and $\Lambda'$ has a good projection.
\end{theorem}

We will prove the above theorem in Section \ref{proof}, following a review of the proof of Legendrian Reidemeister theorems for the front and Lagrangian projections in Section \ref{background}.

Recall that the following lemma computes of the Thurston-Bennequin number of null-homologous Legendrian knots in thickened convex surfaces.

\begin{lemma}[\cite{EaglesRong2026}]
    Let $\Lambda \subset (\Sigma \times \R, \xi_\Gamma)$ be a Legendrian knot. The Thurston-Bennequin number of $\Lambda$ is given by \[\tb(\Lambda) = -\frac{1}{2}|\pi(\Lambda) \cap \Gamma| + \text{writhe}(\pi(\Lambda)).\]
\end{lemma}

Using the Legendrian Reidemeister theorem, one can take the above computation as a definition and check that it is invariant under the Reidemeister moves. Note that for $\tb(\Lambda)$ to be an invariant, it is important that the sign of the crossing in the Reidemeister I move is the one given in Figure \ref{Reidemeister}.

\subsection*{Acknowledgments}
    This paper originated as the first part of the author's thesis project, proposed by his advisor Julian Chaidez. We thank Julian for his guidance in formulating the problem and developing its solution. We thank Marc Kegel for pointing out two mistakes in Lemmas \ref{keyLemma} and \ref{finite_sing_perturbation}, and Marc and Joan Licata for bringing Jack Brand's paper \cite{Brand2023} to our attention. We thank Nancy Eagles, the author's collaborator on the second part of the thesis project, for many helpful discussions arising from our joint work on that part.

\section{Background}\label{background}

In this section we give a rough sketch of proofs of the Legendrian Reidemeister theorem for the front projection \cite{Swiatkowski1992} and Lagrangrian projection in $\R^3$. We also review a result giving a standard neighborhood $N(\Gamma)$ of the dividing curve $\Gamma$.




	




\subsection{Legendrian Reidemeister theorem}

We assume basic familiarity with transversality and the smooth Reidemeister theorem. See for example \cite{GolubitskyGuillemin1973} and \cite{Queffelec2025}\cite[Appendix B]{OSS2015}.

\subsubsection{The front projection}

In this subsection we review Światkowski's proof of the Legendrian Reidemeister Theorem for the front projection \cite{Swiatkowski1992}. See \cite[Appendix B]{OSS2015} for a more detailed exposition.

\begin{definition}
	The \textit{front projection} is $\pi_F: (\R^3, \ker(dz - ydx)) \rightarrow \R^2$ given by $(x, y, z) \mapsto (x, z)$.
\end{definition}

\begin{theorem}
	(a) Let $\Lambda$ be a Legendrian knot in $\R^3$. Then there exists an arbitrarily small Legendrian perturbation $\Lambda'$ of $\Lambda$ such that $\pi_F(\Lambda)$ has only simple cusps (i.e. given by $x_2^2 = x_1^3$ in local coordinates $(x_1, x_2)$) and transverse double points.
	
	(b) Let $\Lambda_t$ be a Legendrian isotopy in $\R^3$ such that $\Lambda_0$ and $\Lambda_1$ satisfy (a). Then there exists an arbitrarily small Legendrian perturbation $\Lambda'_t$ of $\Lambda_t$ such that $\Lambda'_0 = \Lambda_0$, $\Lambda'_1 = \Lambda_1$, and $\pi_F(\Lambda'_t)$ can be discretized into a sequence of Reidemeister moves shown in Figure \ref{front_Reid_moves}.
\end{theorem}

\begin{figure}
    \begin{center}
    	\includegraphics[width=1.0\textwidth]{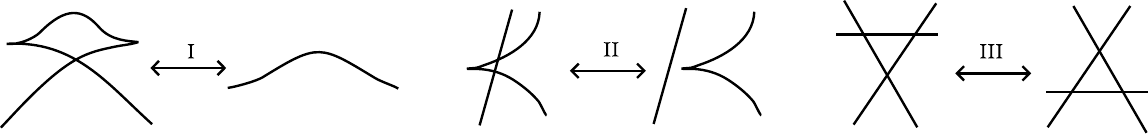}
    	\caption{Reidemeister moves for the front projection}
    	\label{front_Reid_moves}
    \end{center}
\end{figure}

\begin{proof}
	Consider the case of a single Legendrian knot $\Lambda$. One can follow the following procedure to perturb $\Lambda$ so that it is front generic:
    
    \textit{Step-1:} Assume $\Lambda$ is parametrized by $(x(s), y(s), z(s))$. Perturb $x(s)$ into a Morse function $\widetilde{x}(s)$. Note that critical points of $\widetilde{x}(s)$ correspond to cusps in the front projection. Here is a very rough sketch: Assume $\widetilde{x}$ has a Morse critical point at $s = 0$. Ideally $\widetilde{x}(s) = s^2$. Then the contact condition $z' = yx'$ implies $z'(0) = 0$ so $y'(0) \neq 0$. Again ideally $y(s) = \frac{3}{2}s$. So $z'(s) = \frac{3}{2}s \cdot 2s = 3s^2$ and ideally $z(s) = s^3$. This gives the simple cusp $(\widetilde{x}(s), z(s)) = (s^2, s^3)$. In general one can either pick appropriate local coordinates to obtain this local normal form or use an $\epsilon-\delta$ argument as in \cite[Section 3]{Queffelec2025} to obtain a Taylor approximation.

    \textit{Step-2:} Lift the Lagrangian projection $(\widetilde{x}(s), y(s))$ into a Legendrian arc $\widetilde{\Lambda}$ by cutting it at a branch point (i.e., a non-cusp point).

    \textit{Step-3:} In general the Legendrian lift in (b) is not closed at the cutting point, so we need to apply ``branch deformation'' to close up the Legendrian knot. See Figure \ref{branch_deformation} for an illustration of branch deformation and \cite[Section 2]{Swiatkowski1992} for more details. We remark that branch deformation does not change the Legendrian isotopy class of the original knot.

    \begin{figure}
        \begin{center}
        	\includegraphics[width=0.6\textwidth]{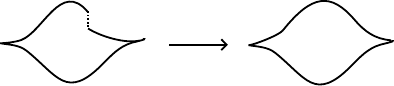}
        	\caption{Branch deformation}
        	\label{branch_deformation}
        \end{center}
    \end{figure}

    Next consider the case of a Legendrian isotopy $\Lambda_t$. The idea of proof is similar to the knot case, except that in a generic $1$-parameter family $\widetilde{x}(s, t) = s^3 + ts$ is possible, which corresponds to a Reidemeister I move for $t \in (-\epsilon, \epsilon)$. This comes from the classical theorem of Whitney: \begin{theorem}
	   A generic map between $2$-manifolds is a local diffeomorphism except that along a $1$-dimensional submanifold it is a fold singularity (i.e. given by $(x_1, x_2) \mapsto (x_1, x_2^2)$ in local coordinates) and along a $0$-dimensional submanifold it is a simple cusp singularity (i.e. given by $(x_1, x_2) \mapsto (x_1, x_2^3 + x_1x_2)$).
    \end{theorem} \noindent See Figure \ref{front_Reid_I}. Here branch deformation needs to be applied $1$-parametrically. Also double and triple intersections are possible, which correspond to the Reidemeister II and III moves. These come from further perturbations of the front projection, which always lifts to Legendrian knots.
\end{proof}

\begin{figure}
    \begin{center}
        \includegraphics[width=1.0\textwidth]{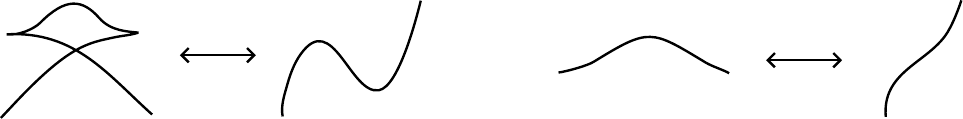}
        \caption{Reidemeister I move for the front projection}
        \label{front_Reid_I}
    \end{center}
\end{figure}

\subsubsection{The Lagrangian projection}

In this subsection we give a proof of the Legendrian Reidemeister theorem for the Lagrangian projection.

\begin{definition}
	The \textit{Lagrangian projection} is $\pi_L: (\R^3, \ker(dz - ydx)) \rightarrow \R^2$ given by $(x, y, z) \mapsto (x, y)$.
\end{definition}

\begin{theorem}
	(a) Let $\Lambda$ be a Legendrian knot in $\R^3$. Then there exists an arbitrarily small Legendrian perturbation $\Lambda'$ of $\Lambda$ such that $\pi_L(\Lambda)$ has only transverse double points.
	
	(b) Let $\Lambda_t$ be a Legendrian isotopy in $\R^3$ such that $\Lambda_0$ and $\Lambda_1$ satisfy (a). Then there exists an arbitrarily small Legendrian perturbation $\Lambda'_t$ of $\Lambda_t$ such that $\Lambda'_0 = \Lambda_0$, $\Lambda'_1 = \Lambda_1$, and $\pi_L(\Lambda'_t)$ can be discretized into a sequence of Reidemeister II, IIIa, and IIIb moves shown in Figure \ref{Reidemeister}.
\end{theorem}

\begin{proof}
	First we consider the case of a single knot $\Lambda$. Fix a parametrization $(x(s), y(s), z(s))$ of $\Lambda$. Perturb $(x(s), y(s))$ smoothly into $(\widetilde{x}(s), \widetilde{y}(s))$ so that it has a good projection, i.e., it has only transverse double points. Fix any $s_0$ such that $(\widetilde{x}(s_0), z(s_0))$ is not a cusp point in the front projection. Then lift the Lagrangian projection to a Legendrian arc $(\widetilde{x}(s), \widetilde{y}(s), \widetilde{z}(s))$ starting from $(\widetilde{x}(s_0), \widetilde{y}(s_0), z(s_0))$. Note that in general it is not a Legendrian knot (i.e. not necessarily closed), so we need to apply branch deformation near $s_0$. This has the effect of adding a little bump to the Lagrangian projection near $(\widetilde{x}(s_0), \widetilde{y}(s_0))$, so alternatively one can mention no branch deformation but only adding small bumps, such as in \cite[Lemma 4.1]{FuchsTabachnikov1997}.

	The above argument works also for a Legendrian isotopy $\Lambda_t$ in which case we need a one-parameter family of branch deformations. Standard genericity arguments in the smooth case give finitely many Reidemeister II/III moves. We claim that Reidemeister I moves are not allowed. Suppose $x'(s_0) = y'(s_0) = 0$. Then $z'(s_0) = y(s_0)x'(s_0) = 0$ so $(x'(s_0), y'(s_0), z'(s_0)) = (0, 0, 0)$, which is a contradiction. Thus as a small perturbation we also have $(\widetilde{x}'(s), \widetilde{y}'(s)) \neq (0, 0)$ for all $s$.
\end{proof}

\subsection{Convex surface theory}

In this subsection we recall some well-known facts of convex surfaces that will be useful later on.

\begin{definition}
	A closed surface $\Sigma$ in a contact $3$-manifold $(M, \xi)$ is \textit{convex} if there exists a contact vector field $X$ near $\Sigma$ transverse to $\Sigma$. $X$ gives a local coordinate $t$ such that there exists a contact form $\alpha$ of the form $\beta + udt$, where $\beta \in \Omega^1(\Sigma)$ and $u: \Sigma \rightarrow \R$. $\Gamma := u^{-1}(0)$ is a union of circles called the \textit{dividing curve}. The dividing curve separates the convex surface into $\Sigma_\pm := u^{-1}(\R_\pm)$.
\end{definition}

\begin{example}
	Consider $\R^3$ with $\xi_{\text{std}} = \ker(dz - ydx)$. Taking $X = \frac{\partial}{\partial x}$, the $yz$-plane is convex with dividing curve $\Gamma = z$-axis and $\Sigma_\pm = \{\mp y > 0\}$. Note that this is the other projection in $\R^3$ apart from the Lagrangian projection ($xy$-projection) and the front projection ($xz$-projection).
\end{example}

The following lemma shows that near the dividing curve $\Gamma$ the convex surface always looks like the above example. 

\begin{lemma}[\text{\cite[Lemma 2.13]{Breen2024}}]\label{stdNbhd}
	Up to a contact isotopy of $\xi$ preserving $\Gamma$ there exists a neighborhood $N(\Gamma) = \Gamma_\theta \times [-\epsilon, \epsilon]_x \times \R_t$ and a contact form $\alpha$ such that $\alpha|_{N(\Gamma)} = d\theta + xdt$.
\end{lemma}

%
%

 
\section{Legendrian Reidemeister theorem for the convex surface projection}\label{proof}

In this section we prove the Legendrian Reidemeister Theorem for the convex surface projection. 



The following lemma tells us how to identify singularities along $\Gamma$ for the convex surface projection $\pi$ in terms of the front projection $\pi_F$.

\begin{figure}
    \begin{center}
    	\includegraphics[width=1.0\textwidth]{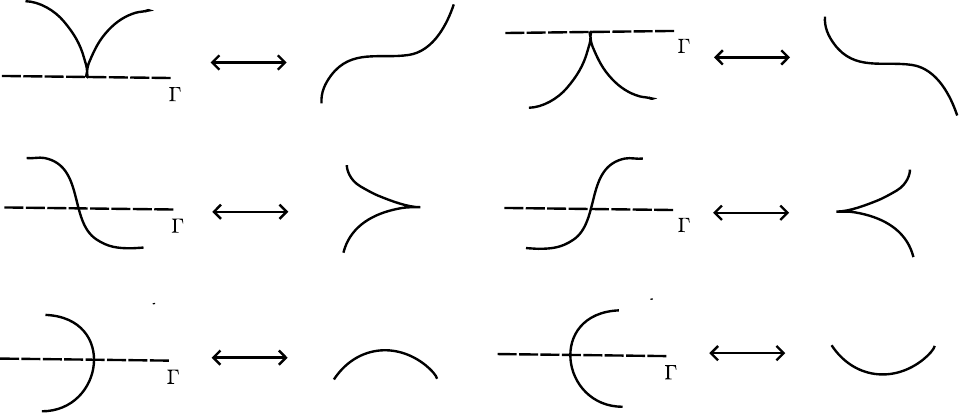}
    	\caption{Correspondence between singularities in $\pi$ and $\pi_F$}
    	\label{Key_lemma}
    \end{center}
\end{figure}

\begin{lemma}\label{keyLemma}
	Here we switch the notation from $(t, x, \theta)$ to $(x, y, z)$ so that the contact form is $dz + ydx$ and we view the front projection $\pi_F(\Lambda_t)$ as graphs of functions of $x$. Assume that $\Lambda_t$ is front generic.
	
	(a) Non-immersion points of $\pi|_{\Lambda_t}$ correspond to non-Morse (and non-cusp) critical points of $\pi_F(\Lambda_t)$. 
	
	(b) Transverse intersections of $\pi(\Lambda_t)$ with $\Gamma$ correspond to Morse critical points (including slope-$0$ cusp points) of $\pi_F(\Lambda_t)$.
	
	See Figure \ref{Key_lemma}.
\end{lemma}

\begin{proof}
	First we prove (a). Assume that $(x, y, z)$ is a non-immersion point, i.e., $y' = z' = 0$. Since $z' = -yx'$ and $x' \neq 0$, then $y = 0$. So \[z'' = -y'x' - yx'' = 0.\] Also $x' \neq 0$ implies $(x, y, z)$ is not a cusp point in the front projection. Conversely, assume $z' = z'' = 0$. Since $x' \neq 0$, then $y = 0$. So $z'' = -y'x'$ implies $y' = 0$ and $(x, y, z)$ is a non-immersion point.
	
	Next we prove (b). Indeed, in this case we have $z' = 0$ and $y' \neq 0$. If $x' = 0$, then $(x, y, z)$ is a cusp point in the front projection. Moreover, it has slope $0$ since $y = 0$ along $\Gamma$. Otherwise \[z'' = -y'x' \neq 0\] and $(x, y, z)$ is a Morse critical point. Conversely, assume $z' = 0$ and $z'' \neq 0$. Assume $y = 0$. Then $z'' = - y'x' - yx'' = -y'x'$ implies $y' \neq 0$ so we have a transverse intersection with $\Gamma$. Otherwise $y \neq 0$ so $x' = 0$ and $(x, y, z)$ is a cusp point. This is a contradiction since $y = 0$ at a slope-$0$ cusp point. 
\end{proof}

The following lemma perturbs $\Lambda_t$ so that it has only finitely many singularities.

\begin{lemma}\label{finite_sing_perturbation}
	There exists a Legendrian perturbation of $\Lambda_t$ such that: 
	
	(a) $\pi|_{\Lambda_t}$ is an immersion except that $\Lambda_t$ can have a cusp along $\Gamma$ for finitely many $t$.
	
	(b) $\pi(\Lambda_t)$ has only finitely many double points along $\Gamma$.
	
	(c) $\pi(\Lambda_t)$ has no triple points along $\Gamma$.
	
	(d) $\pi(\Lambda_t)$ has only finitely many double points away from $\Gamma$, all of which are tangent.
	
	(e) $\pi(\Lambda_t)$ has only finitely many triple points away from $\Gamma$, all of which are transverse.
	
	(f) $\pi(\Lambda_t)$ has no quadruple points away from $\Gamma$.
\end{lemma}

\begin{proof}
    First we would like to clarify how we perturb $\pi_F(\Lambda_t)$. The problem with perturbation in the usual sense is that generically there cannot be a $1$-parameter family of cusps in an isotopy (cusps only appear as the bifurcation point of a Reidemeister I move). In \cite{Swiatkowski1992}, perturbation in the front projection of Legendrian knots in $\R^3$ is done using branch deformation and generalized branch deformation (i.e., branch deformation that moves the cusps while keeping the local ``shape'' of cusps fixed). Here since we would need to change the shape of cusps as well, we introduce an alternative definition of fronts, ``smoothing'' the cusps so that they are not destroyed by a generic perturbation in the usual sense. We begin with the following definition:

    \begin{figure}
        \begin{center}
        	\includegraphics[width=0.2\textwidth]{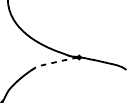}
        	\caption{Alternative definition of a front}
        	\label{front transversality}
        \end{center}
    \end{figure}

    \begin{definition}
        A \textit{front} in $\R^2$ is a set of (disconnected) smooth (without cusps) arcs together with a choice of finitely many points, such that local vertical reflections at the points give genuine fronts in the usual sense, with the chosen points being the cusp points of the front. See Figure \ref{front transversality}.
    \end{definition}

    With the above definition, we can do usual perturbation without killing the cusps (since that are smooth now), at the cost that an arbitrary perturbation may lead to a disconnected front, which we then close up using branch deformation. Later in the proof we will perturb $\pi_F(\Lambda_t)$ in this new sense.
    
	Now we give a sketch of the overall argument. Pick a neighborhood $N(\Gamma)$ of $\Gamma$ and isotope $\xi$ so that $\alpha$ satisfies the conditions of Lemma \ref{stdNbhd}. We use the notation $\Lambda_t(r)$ to denote $\Lambda_t \cap (\Gamma \times [-r, r] \times \R)$. In particular, $N(\Gamma) = \Lambda_t(\epsilon)$. First perturb $\Lambda_t(\epsilon)$ so that $\Lambda_t(\frac{3\epsilon}{4})$ has a good front projection. Then perturb $\Lambda_t(\frac{3\epsilon}{4})$ so that $\Lambda_t(\frac{\epsilon}{2})$ is good, i.e., (a)-(f) are all satisfied. Since by rescaling $\alpha$ by a positive function we can make $\alpha$ take the form $\beta \pm dt$ on $\Lambda_t - \Lambda_t(\frac{\epsilon}{4})$, we can perturb $\Lambda_t - \Lambda_t(\frac{\epsilon}{4})$ similarly as for the Lagrangian projection so that $\Lambda_t - \Lambda_t(\frac{\epsilon}{2})$ is good, i.e., (d)(e)(f) are satisfied. Now the entire $\Lambda_t$ satisfies (a)-(f). For the rest of the proof we justify that each of (a)-(f) can indeed be achieved using the Thom transversality theorem.
	
	Write $\Lambda_t$ as $f: \Lambda \times [0,1] \rightarrow N(\Gamma) \times \R$. The key is to perturb $\pi_F \circ f$ so that $j^1_s (\pi_F \circ f)$ is transverse to a list of submanifolds of $J^k_s(S^1 \times [0, 1], \R^2)$ for some $k, s$, one for each case from (a)-(f). These perturbations lift to Legendrian perturbations of $f$ since front diagrams always lift to Legendrians and perturbations of Lagrangian diagrams can always be made Legendrian by branch deformation or adding small bumps.
	
	By Cerf theory a generic $1$-parameter family of functions $\R \rightarrow \R$ is Morse except at finitely many points where a birth/death of a pair of Morse critical points occur. Indeed, the submanifold \[\{(s, t, \begin{pmatrix}
		x \\
		z \\
	\end{pmatrix}, \begin{pmatrix}
		x_{s} & x_{t} \\
		z_{s} & z_{t} \\
	\end{pmatrix}, \begin{pmatrix}
		x_{s^2} & x_{st} & x_{t^2} \\
		z_{s^2} & z_{st} & z_{t^2} \\
	\end{pmatrix}) : z_s = z_{s^2} = 0\} \subset J^2(S^1 \times [0,1], \R^2)\] has codimension $2$. So by a generic perturbation of the front projection we can achieve that $\pi_F(\Lambda_t)$ has finitely many non-Morse critical points, i.e. by Lemma \ref{keyLemma} $\pi(\Lambda_t)$ has finitely many non-immersion points. This proves (a). 
	
	Two transverse intersections with $\Gamma$ coincide exactly when they share the same $z$-coordinate. By Lemma \ref{keyLemma}, the relevant submanifold is \[\{(s_i, t_i, \begin{pmatrix}
		x_i \\
		z_i \\
	\end{pmatrix}, \begin{pmatrix}
		x_{s_i} & x_{t_i} \\
		z_{s_i} & z_{t_i} \\
	\end{pmatrix}, \begin{pmatrix}
		x_{s_i^2} & x_{s_i t_i} & x_{t_i^2} \\
		z_{s_i^2} & z_{s_i t_i} & z_{t_i^2} \\
	\end{pmatrix})_{i = 1, 2} \in J^2_2(S^1 \times [0,1], \R^2): \] \[t_1 = t_2, z_{s_1} = z_{s_2} = 0, z_{s_1^2} \neq 0, z_{s_2^2} \neq 0, z_1 = z_2\},\] which has codimension $4$. This proves (b).
	
	The condition that three critical points in $\pi_F(\Lambda_t)$ share the same $z$-coordinate corresponds to the submanifold \[\{(s_i, t_i, \begin{pmatrix}
		x_i \\
		z_i \\
	\end{pmatrix}, \begin{pmatrix}
		x_{s_i} & x_{t_i} \\
		z_{s_i} & z_{t_i} \\
	\end{pmatrix})_{i = 1}^3 \in J^1_3(S^1 \times [0, 1], \R^2) : t_1 = t_2 = t_3,\] \[z_{s_1} = z_{s_2} = z_{s_3} = 0, z_{s_1^2} \neq 0, z_{s_2^2} \neq 0, z_{s_3^2} \neq 0, z_1 = z_2 = z_3\},\] which has codimension $7 > 6$. This proves (c).
	
	(d)(e)(f) follow from the Legendrian Reidemeister theorem for the Lagrangian projection.
\end{proof}


The following lemma further perturbs $\Lambda_t$ so that each singularity gives a Reidemeister move and this finishes the proof of Theorem \ref{main}.

\begin{lemma}
	We can further perturb $\Lambda_t$ so that
	
	(a) Each non-immersion point along $\Gamma$ corresponds to a Reidemeister I move along $\Gamma$.
	
	(b) Each double point along $\Gamma$ corresponds to a Reidemeister IV move or a Reidemeister IV move followed by a Reidemeister II move.
	
	(c) Each double point away from $\Gamma$ corresponds to a Reidemeister II move.
	
	(d) Each triple point away from $\Gamma$ corresponds to a Reidemeister III move.
\end{lemma}

\begin{figure}
    \begin{center}
    	\includegraphics[width=1.0\textwidth]{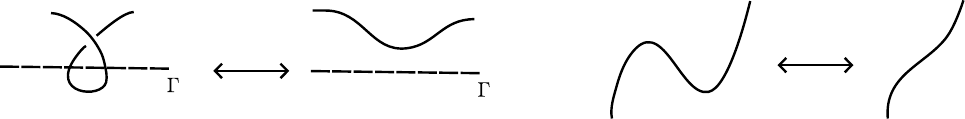}
    	\caption{Reidemeister I move}
    	\label{Reidemeister I}
    \end{center}
\end{figure}

\begin{figure}
    \begin{center}
    	\includegraphics[width=1.0\textwidth]{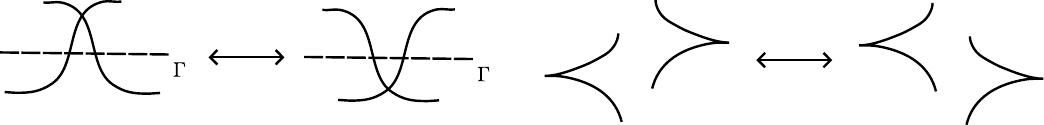}
    	\caption{An Example of a Reidemeister IV move}
    	\label{Reidemeister IV_2}
    \end{center}
\end{figure}

\begin{figure}
    \begin{center}
    	\includegraphics[width=1.0\textwidth]{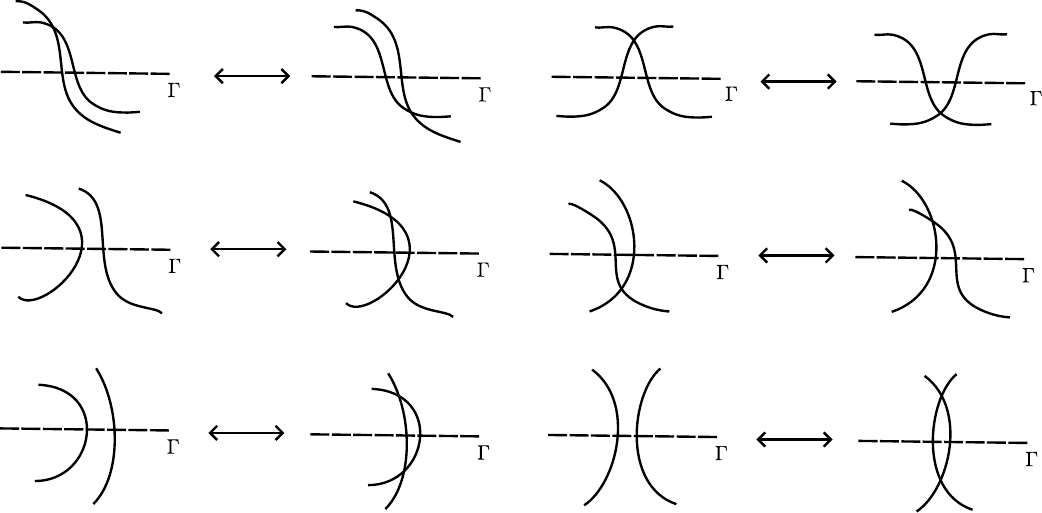}
    	\caption{Reidemeister IV moves and composition of Reidemeister II and IV moves}
    	\label{Reidemeister IV}
    \end{center}
\end{figure}

\begin{proof}
	For (a), genericity of the non-Morse critical points will guarantee genericity of the non-immersion points, i.e., they are simple cusps near which a standard Reidemeister I move occurs. See Figure \ref{Reidemeister I}. By computing an explicit model one can check that the sign of the crossing has to be the one drawn in the figure. For (b), the two possibilities mentioned in the statement depend on whether the two intersections with $\Gamma$ are (in the front projection) left/right cusps or local minima/maxima. See Figure \ref{Reidemeister IV_2} for an example and \ref{Reidemeister IV} for all possible moves corresponding to a double point along $\Gamma$. (c) and (d) follow from the Legendrian Reidemeister theorem for the Lagrangian projection.
\end{proof}

%

\printbibliography

\end{document}